\documentclass[12pt]{amsart}

\usepackage{amsmath}
\usepackage{amssymb}
\usepackage{amsfonts}
\usepackage{amsthm}
\usepackage{enumerate}
\usepackage{hyperref}
\usepackage{color}

\theoremstyle{plain}
\newtheorem{thm}{Theorem}[section]

\newtheorem{prop}[thm]{Proposition}
\newtheorem{lem}[thm]{Lemma}
\newtheorem{cor}[thm]{Corollary}

\theoremstyle{definition}
\newtheorem{dfn}[thm]{Definition}

\newtheorem{rem}[thm]{Remark}

\newtheorem{dfns-rems}[thm]{Definitions and Remarks}
\newtheorem{notas-rems}[thm]{Notations and Remarks}
\newtheorem{exmps-rems}[thm]{Examples and Remarks}

\begin{document}


\title[Symbolic powers of cover ideals]{Symbolic powers of cover ideal of very well-covered and bipartite graphs}


\author[S. A. Seyed Fakhari]{S. A. Seyed Fakhari}

\address{S. A. Seyed Fakhari, School of Mathematics, Statistics and Computer Science,
College of Science, University of Tehran, Tehran, Iran.}

\email{fakhari@khayam.ut.ac.ir}

\urladdr{http://math.ipm.ac.ir/$\sim$fakhari/}


\begin{abstract}
Let $G$ be a graph with $n$ vertices and $S=\mathbb{K}[x_1,\dots,x_n]$ be the
polynomial ring in $n$ variables over a field $\mathbb{K}$. Assume that $J(G)$ is the cover ideal of $G$ and $J(G)^{(k)}$ is its $k$-th symbolic power. We prove that if $G$ is a very well-covered graph such that $J(G)$ has linear resolution, then $J(G)^{(k)}$ has linear resolution, for every integer $k\geq 1$. We also prove that for a every very well-covered graph $G$, the depth of symbolic powers of $J(G)$ forms a non-increasing sequence. Finally, we determine a linear upper bound for the regularity of powers of cover ideal of bipartite graph.
\end{abstract}


\subjclass[2000]{Primary: 13D02, 05E99; Secondary: 13C15}


\keywords{Cover ideal, Very well-covered graph, Linear resolution, Regularity}


\thanks{}


\maketitle


\section{Introduction} \label{sec1}

Over the last 25 years the study of algebraic, homological and combinatorial
properties of powers of ideals has been one of the major topics in Commutative
Algebra. In this paper we study the minimal free resolution of the powers of cover ideal of graphs. Cover ideal of a graph is the Alexander dual of its edge ideal and has been studied by several authors (see e.g., \cite{bbv}, \cite{cpsty}, \cite{fhv}, \cite{fhv1}, \cite{m}).

In Section \ref{sec2}, we study the minimal free resolution of symbolic powers of cover ideal of very well-covered graphs. A graph $G$ is said to be very well-covered if the cardinality of every maximal independent set of $G$ is half of the number of vertices of $G$. The family of very well-covered graphs is rich, because it includes all the unmixed bipartite graphs, which have no isolated vertex. This class of graphs is studied from algebraic point of view in \cite{crt}, \cite{kty} \cite{mmcrty}. Let $G$ be a graph with $n$ vertices and $S=\mathbb{K}[x_1,\dots,x_n]$ be the
polynomial ring in $n$ variables over a field $\mathbb{K}$. Assume that $J(G)$ is the cover ideal of $G$. The first main result of Section \ref{sec2} determines a class of monomial ideals, such that all symbolic powers of an ideal in this class have linear resolution. In fact, There are many attempts to characterize the monomial ideals with linear resolution. One of the most important results in this direction is due to Fr${\rm \ddot{o}}$berg \cite[Theorem 1]{f}, who characterized all squarefree monomial ideals generated by quadratic monomials, which have linear resolution. It is also known \cite{ht} that polymatroidal ideals have linear resolution and that powers
of polymatroidal ideals are again polymatroidal (see \cite{hh}). In particular
they have again linear resolution. In general however, powers of ideals with linear
resolution need not to have linear resolution. The first example of such an ideal
was given by Terai. He showed that over a base field of characteristic $\neq 2$ the Stanley
Reisner ideal $I = (x_1x_2x_3, x_1x_2x_5, x_1x_3x_6, x_1x_4x_5, x_1x_4x_6, x_2x_3x_4, x_2x_4x_6, x_2x_5x_6, x_3x_4x_5, x_3x_5x_)$ of the minimal
triangulation of the projective plane has linear resolution, while $I^2$ has not linear
resolution. This example depends on the characteristic of the base field. If the base
field has characteristic $2$, then $I$ itself has not linear resolution.
Another example, namely $I = (x_4x_5x_6, x_3x_5x_6, x_3x_4x_6, x_3x_4x_5, x_2x_5x_6, x_2x_3x_4, x_1x_3x_6, x_1x_4x_5)$ is given by
Sturmfels \cite{s}. Again $I$ has linear resolution, while $I^2$ has not linear resolution. However, Herzog, Hibi and Zheng \cite{hhz} prove that a monomial ideal $I$ generated in degree $2$ has linear resolution if and only if every power of $I$ has linear resolution. Also, it follows from \cite[Theorem 2.2]{mm} that if $G$ is a bipartite graph such that $J(G)$ has linear resolution, then every power of $J(G)$ has linear resolution too. Our Theorem \ref{main} is a generalization of this result and asserts that if $G$ is a very well-covered graph, such that $J(G)$ has linear resolution, then for every integer $k\geq 1$, the $k$-th symbolic power of $J(G)$, denoted by $J(G)^{(k)}$, has linear resolution and even more, it has linear quotients. In order to prove this result, in Proposition \ref{very-well}, we introduce a construction to obtain a Cohen--Macaulay very well-covered graph from a given one. We will see that the cover ideal of the resulting graph is related to the symbolic powers of the cover ideal of the primary graph, via polarization. In Corollary \ref{linbi}, we prove that the converse of Theorem \ref{main} is true for bipartite graphs. In other words, for a bipartite graph $G$, the cover ideal $J(G)$ has linear resolution if and only if $J(G)^{(k)}$ has linear resolution for some integer $k\geq 1$. Next, in Corollary \ref{sdepth}, we prove that if $G$ is a very well-covered graph such that $J(G)$ has linear resolution, then for every integer $k\geq 1$, the modules $J(G)^{(k)}$ and $S/J(G)^{(k)}$ satisfy Stanley's inequality, i.e., their Stanley depth is an upper bound for their depth. In the proof of Corollary \ref{sdepth}, we use the result obtained in \cite{mmcrty}, which states that for very well-covered graphs the notions of Cohen--Macaulayness and vertex decomposability are the same. As last result of Section \ref{sec2}, we study the depth of the symbolic powers of the cover ideal of a very well-covered graph. In \cite[Theorem 3.2]{cpsty}, it is shown that the depth of the symbolic powers of the cover ideal of a bipartite graph is a non-increasing sequence. In Theorem \ref{depthsym}, we prove the same for every very well-covered graphs.

Computing and finding bounds for the regularity of powers of a monomial ideal have been studied by a number of researchers (see for example \cite{b''}, \cite{bht}, \cite{c'}, \cite{htt}). It is well-known that ${\rm reg}(I^{s})$ is asymptotically a linear function for $s\gg0$. However, it is usually difficult to compute this linear function or estimate it. In Section \ref{sec3}, we study the regularity of (ordinary) powers of cover ideal of a bipartite graph (note that by \cite[Corollary 2.6]{grv}, for the cover ideal of bipartite graphs the ordinary and symbolic powers coincide). In Theorem \ref{regbi}, we determine a linear upper bound for the regularity of these ideals. More explicit, we prove that for a bipartite graph $G$ and every integer $k\geq 1$, the regularity of $S/J(G)^k$ is at most $k{\rm deg}(J(G))+{\rm reg}(S/J(G))-1$, where for a monomial ideal $I$, we denote the maximum degree of minimal monomial  generators of $I$ by ${\rm deg}(I)$.


\section{Preliminaries} \label{sec1'}

In this section, we provide the definitions and basic facts which will be used in the next sections. We refer the reader to \cite{hh} for undefined terminologies.

Let $G$ be a simple graph with vertex set $V(G)=\big\{x_1, \ldots,
x_n\big\}$ and edge set $E(G)$ (by abusing the notation, we identify the vertices of $G$ with the variables of $S$). For a vertex $x_i$, the {\it neighbor set} of $x_i$ is $N_G(x_i)=\{x_j\mid x_ix_j\in E(G)\}$ and We set $N_G[x_i]=N_G(x_i)\cup \{x_i\}$ and call it the {\it closed neighborhood} of $x_i$. For a subset $F\subseteq V(G)$, we set $N_G[F]=\cup_{x_i\in F}N_G[x_i]$. For every subset $A\subset V(G)$, the graph $G\setminus A$ is the graph with vertex set $V(G\setminus A)=V(G)\setminus A$ and edge set $E(G\setminus A)=\{e\in E(G)\mid e\cap A=\emptyset\}$. A subgraph $H$ of $G$ is called induced provided that two vertices of $H$ are adjacent if and only if they are adjacent in $G$. A {\it matching} in a graph $G$ is a subgraph consisting of pairwise disjoint edges. If the subgraph is an induced subgraph, the matching is an {\it induced matching}. The cardinality of the maximum induced matching of $G$ is denoted by ${\rm indmatch}(G)$. A subset $W$ of $V(G)$ is called an {\it independent subset} of $G$ if there are no edges among the vertices of $W$. The graph $G$ is said to be {\it very well-covered} if $n$ is an even number and every maximal independent subset of $G$ has cardinality $n/2$. The {\it independence simplicial complex} of $G$ is defined by
$$\Delta(G)=\{A\subseteq V(G)\mid A \,\, \mbox{is an independent set in}\,\,
G\}.$$Note that Stanley--Reisner ideal of $\Delta(G)$ is the {\it edge ideal} of $G$ which is defined to be$$I(G)=(x_ix_j\mid \{x_1, x_j\}\in E(G))\subset S.$$
A subset $C$ of $V(G)$ is called a {\it vertex cover} of the graph $G$ if every edge of $G$ is incident to at least one vertex of $C$. A vertex cover $C$ is called a {\it minimal vertex cover} of $G$ if no proper subset of $C$ is a vertex cover of $G$. Note that $C$ is a minimal
vertex cover if and only if $V(G)\setminus C$ is a maximal independent set,
that is, a facet of $\Delta(G)$. A graph $G$ is called {\it unmixed} if all
minimal vertex covers of $G$ have the same number of elements. The size of the smallest vertex cover of $G$ will be denoted by $\beta'(G)$. The Alexander dual of the edge ideal of $G$ in $S$, i.e., the
ideal $$J(G)=I(G)^{\vee}=\bigcap_{\{x_i,x_j\}\in E(G)}(x_i,x_j),$$ is called the
{\it cover ideal} of $G$ and is the main objective of this paper. The reason for this name is due to the
well-known fact that the generators of $J(G)$ correspond to minimal vertex covers of $G$.

Let $I$ be a squarefree monomial ideal in $S$ and suppose that $I$ has the irredundant
primary decomposition $$I=\frak{p}_1\cap\ldots\cap\frak{p}_r,$$ where every
$\frak{p}_i$ is an ideal of $S$ generated by a subset of the variables of
$S$. Let $k$ be a positive integer. The $k$th {\it symbolic power} of $I$,
denoted by $I^{(k)}$, is defined to be $$I^{(k)}=\frak{p}_1^k\cap\ldots\cap
\frak{p}_r^k.$$

For a graded $S$-module $M$, we denote the graded Betti numbers of $M$ by $\beta_{i,j}(M)$. The Castelnuovo--Mumford regularity (or simply, regularity) of $M$, denote by ${\rm reg}(M)$, is defined as follows.
$${\rm reg}(M)=\max\{j-i|\ \beta_{i,j}(M)\neq0\}.$$The module $M$ is said to
have {\it linear resolution}, if for some integer $d$, $\beta_{i,i+t}(M)=0$
for all $i$ and every $t\neq d$. It is clear from the definition that if a monomial ideal has a linear resolution, then all the minimal monomial generators of $I$ have the same degree. Next, we recall the definition of monomial ideals with linear quotients. We remind that for a monomial ideal $I$, the set of minimal monomial generators of $I$ is denoted by $G(I)$.

\begin{dfn}
Let $I$ be a monomial ideal. Assume that $u_1\prec u_2 \prec \ldots \prec u_m$ is a linear order on $G(I)$. We say that $I$ has {\it linear quotient with respect to $\prec$}, if for every $2\leq i\leq m$, the ideal $(u_1, \ldots, u_{i-1}):u_i$ is generated by a subset of variables. We say that $I$ has {\it linear quotient}, if it has linear quotient with respect to a linear order on $G(I)$.
\end{dfn}

It is known that if $I$ is a monomial ideal which is generated is a sing degree and has linear quotients, then it admits linear resolution. Monomial ideals with linear quotients are related to an important class of simplicial complexes, namely shellable simplicial complexes.

\begin{dfn}
A simplicial complex $\Delta$ is called
{\it shellable} if its facets can be arranged in linear order $F_1,F_2,
\ldots,F_t$ in such a way that the subcomplex $\langle F_1, \ldots,F_{k-1}
\rangle \cap \langle F_k\rangle$ is pure and has dimension $\dim F_k-1$ for
every $k$ with $2\leq k\leq t$.
\end{dfn}

By \cite[Theorem 8.2.5]{hh}, a simplicial complex $\Delta$ is shellable if and only if $I_{\Delta^{\vee}}$ has linear quotients, where $\Delta^{\vee}$ is the Alexander dual of $\Delta$. A simplicial complex $\Delta$ is called to be Cohen--Macaulay if its Stanley--Reisner ring $\mathbb{K}[\Delta]:=S/I_{\Delta}$ is a Cohen--Macaulay ring. A fundamental result in combinatorial commutative algebra says that a pure shellabe simplicial complex is Cohen--Macaulay. Also, it follows from Eagon--Reiner Theorem \cite[Theorem 8.1.9]{hh}, that a simplicial complex $\Delta$ is Cohen--Macaulay if and only if $I_{\Delta^{\vee}}$ has linear resolution.

Let $\Delta$ be a simplicial complex. The {\it link}
of $\Delta$ with respect to a face $F \in \Delta$, denoted by ${\rm lk_
{\Delta}}(F)$, is the simplicial complex ${\rm lk_{\Delta}}(F)=\{G
\subseteq [n]\setminus F\mid G\cup F\in \Delta\}$ and the {\it deletion} of
$F$, denoted by ${\rm del_{\Delta}}(F)$, is the complex ${\rm
del_{\Delta}}(F)=\{G \subseteq [n]\setminus F\mid G \in \Delta\}$. When $F
= \{x\}$ is a single vertex, we abuse notation and write ${\rm lk_{\Delta
}}(x)$ and ${\rm del_{\Delta}}(x)$. We are now ready to define vertex decomposable simplicial complexes which will be used in the proof of Corollary \ref{sdepth}.

\begin{dfn}
Let $\Delta$ be a simplicial complex. Then we say that $\Delta$ is {\it vertex decomposable} if either
\begin{itemize}
\item[(1)] $\Delta$ is a simplex, or\\[-0.3cm]

\item[(2)] $\Delta$ has a vertex $x$ such that ${\rm del_{\Delta}}(x)$ and
    ${\rm lk_{\Delta}}(x)$ are vertex decomposable and every facet of
    ${\rm del_{\Delta}}(x)$ is a facet of $\Delta$.
\end{itemize}
\end{dfn}

A graph $G$ is called Cohen--Macaulay/shellable/vertex decomposable if $\Delta(G)$ has the same property. Thus, the graph $G$ is Cohen--Macaulay if and only if $J(G)$ has linear resolution and it is shellable if and only if $J(G)$ has linear quotients. We know from \cite[Theorem 1.1]{mmcrty} that for very well-covered graphs, the concepts of Cohen--Macaulayness, shellability and vertex decomposability are equivalent.


\section{Very well-covered graphs} \label{sec2}

The aim of this section is to study the minimal free resolution of symbolic powers of cover ideal of very well-covered graphs. We first prove in Theorem \ref{main} that if $G$ is a very well-covered graph such that $J(G)$ has linear resolution, then every symbolic power $J(G)^{(k)}$ has linear resolution too. In order to prove this result, we introduce a construction to obtain a Cohen--Macaulay very well-covered graph from a given one. In the following construction, for every graph $G$ and every integer $k\geq 1$, we build a new graph $G_k$ whose cover ideal is strongly related to the $k$-th symbolic power of $J(G)$ (see Lemma \ref{sympol}).

{\bf Construction.} Let $G$ be a graph with vertex set $V(G)=\{x_1, \ldots, x_n\}$ and let $k\geq 1$ be an integer. We define the new graph $G_k$ on new vertices $$V(G_k)=\{x_{i,p}\mid 1\leq i\leq n \ {\rm and} \  1\leq p\leq k\},$$(thus $G_k$ has $nk$ vertices) and the edge set of $G_k$ is $$E(G_k)=\{\{x_{i,p}, x_{j,q}\}\mid \{x_i, x_j\}\in E(G) \  {\rm and} \  p+q\leq k+1\}.$$

\begin{prop} \label{very-well}
Let $G$ be a graph without isolated vertices and $k\geq 1$ be an integer.
\begin{itemize}
\item[(a)] If $G$ is very well-covered, then $G_k$ is very well-covered too.
\item[(b)] If $G$ is Cohen--Macaulay and very well-covered, then $G_k$ is Cohen--Macaulay too.
\end{itemize}
\end{prop}

\begin{proof}
Since $G$ is very well-covered, $n=|V(G)|$ is an even integer. Set $h=n/2$. By \cite[Lemma 4.1]{mmcrty} and \cite[Proposition 2.3]{crt}, the vertices of $G$ can be relabeled, say $V(G)=\{w_1, \ldots, w_h, z_1, \ldots, z_h\}$ such that
\begin{itemize}
\item[(i)] $\{w_1, \ldots, w_h\}$ is a minimal vertex cover of $G$ and $\{z_1, \ldots, z_h\}$ is a maximal independent set of $G$;
\item[(ii)] $\{w_1, z_1\}, \ldots, \{w_h, z_h\}\in E(G)$;
\item[(iii)] if $\{y_i, w_j\}, \{z_j, w_l\}\in E(G)$, then $\{y_i, w_l\}\in E(G)$ for distinct $i, j, l$ and for $y_i\in \{w_i, z_i\}$;
\item[(iv)] if $\{w_i, z_j\}\in E(G)$, then $\{w_i, w_j\}\notin E(G)$.
\end{itemize}
We rename the vertices of $G_k$ as follows.
\begin{align*}
& a_1:=w_{1,1}, \ a_2:=w_{2,1}, \ldots, a_h:=w_{h,1}, \ a_{h+1}:=w_{1,2}, \ldots, a_{2h}:=w_{h,2}, \ldots, a_{kh}:=w_{h,k},\\ & b_1:=z_{1,k}, \ b_2:=z_{2,k}, \ldots,  b_h:=z_{h,k}, \ b_{h+1}:=z_{1,k-1}, \ldots, b_{2h}:=z_{h,k-1}, \ldots, b_{kh}:=z_{h,1}.
\end{align*}
It is clear from (i) and the construction of $G_k$ that $\{a_1, a_2, \ldots, a_{kh}\}$ is a minimal vertex cover of $G_k$ and $\{b_1, b_2, \ldots, b_{kh}\}$ is a maximal independent set of $G_k$. This shows that $\beta'(G_k)\leq kh$. Also, it follows from (ii) that $\{\{a_1, b_1\}, \ldots, \{a_{kh},b_{kh}\}\}$ is a perfect matching of $G_k$. Therefore $\beta'(G_k)=kh$. This implies that ${\rm ht}(I(G_k))=kh$.

Assume that $i, j, l$ are distinct integers with $1\leq i, j, l\leq kh$. Then there exist integers $m,p,q,r,s,t$ such that $a_i=w_{m,p}, b_i=z_{m, k+1-p}, a_j=w_{q,r}, b_j=z_{q, k+1-r}$ and $a_l=w_{s,t}$. We continue the proof in several steps.\\

{\bf Step 1.} If $\{a_i, a_j\}, \{b_j, a_l\}\in E(G_k)$, then $\{a_i,a_l\}\in E(G_k)$.

{\it Proof.} It follows from the assumptions that $\{w_m, w_q\}, \{w_s, z_q\}\in E(G), \ p+r\leq k+1$ and $k+1-r+t\leq k+1$. Hence, $p+t=(p+r)+(k+1-r+t)-(k+1)\leq k+1$. Since $\{w_m, w_q\}\in E(G)$, we conclude that $m\neq q$. It also follows from (iv) that $s\neq m$. If $s=q$, then $\{w_m, w_s\}=\{w_m, w_q\}\in E(G)$. Thus, $\{a_i,a_l\}$ is an edge of $G_k$. If $s\neq q$, then it follows from (iii) that $\{w_m, w_s\}$ is an edge of $G$. Thus, again $\{a_i,a_l\}\in E(G_k)$.\\

{\bf Step 2.} If $\{b_i, a_j\}, \{b_j, a_l\}\in E(G_k)$, then $\{b_i, a_l\}\in E(G_k)$.

{\it Proof.} It follows from the assumptions that $\{z_m, w_q\}, \{w_s, z_q\}\in E(G), \ k+1-p+r\leq k+1$ and $k+1-r+t\leq k+1$. Hence, $(k+1-p)+t=(k+1-p+r)+(k+1-r+t)-(k+1)\leq k+1$. If $s=q$, then it follows from $\{z_m, w_s\}=\{z_m, w_q\}\in E(G)$ that $\{b_i, a_l\}\in E(G_k)$. Therefore, assume that $s\neq q$. Similarly, we can assume that $m\neq q$. If $s=m$, it follows from (ii) that $\{z_m, w_s\}=\{z_m, w_m\}\in E(G)$, which implies that $\{b_i, a_l\}\in E(G_k)$. Thus, suppose that $s\neq m$.  Hence, $m,q$ and $s$ are distinct. Then (iii) implies that $\{z_m, w_s\}$ is an edge of $G$. Thus, again $\{a_i,a_l\}\in E(G_k)$.\\

{\bf Step 3.} If $\{a_i, b_j\}\in E(G_k)$, then $\{w_m,z_q\}\in E(G)$ and it follows from (iv) that $\{w_m,w_q\}\notin E(G)$. Thus, $\{a_i, a_j\}\notin E(G_k)$.\\

Now, Steps 1, 2, 3 and \cite[Theorem 2.9]{mrv} (see also \cite[Proposition 2.3]{crt}) imply that $G$ is unmixed and since ${\rm ht}(I(G_k))=kh=|V(G)|/2$, we conclude that $G_k$ is a very well-covered graph.

Next assume that $G$ is a Cohen--Macaulay very well-covered graph. By \cite[Lemma 3.1]{mmcrty} there is a relabeling for the vertices of $G$ which satisfies conditions (i)-(iv), mentioned above and the following condition.
\begin{itemize}
\item[(v)] If $\{w_i, z_j\}\in E(G)$, then $i\leq j$.
\end{itemize}

{\bf Step 4.} If $\{a_i,b_j\}\in E(G_k)$, then $i\leq j$.

{\it Proof.} It follows from the assumption that $\{w_m, z_q\}\in E(G)$ and $p+(k+1-r)\leq k+1$. Thus $p\leq r$ and It follow from (v) that $m\leq q$. Therefor $i\leq j$.

Finally, it follows from Steps 1, 2, 3, 4 and \cite[Lemma 3.1]{mmcrty} that $G_k$ is Cohen--Macaulay.
\end{proof}

\begin{rem}
Although we proved in Proposition \ref{very-well} that for every very well-covered graph $G$, the graph $G_k$ is also very well-covered graph, but it is not in general true that $G_k$ is unmixed if $G$ is. For example, let $G=K_3$ be the complete graph with three vertices. Then $G$ is unmixed (even Cohen--Macaulay) but $G_2$ is not unmixed.
\end{rem}

We next remind the definition of polarization. It is a very useful machinery to convert a monomial ideal to a squarefree one.

\begin{dfn} \label{pol}
Let $I$ be a monomial ideal of
$S=\mathbb{K}[x_1,\ldots,x_n]$ with minimal generators $u_1,\ldots,u_m$,
where $u_j=\prod_{i=1}^{n}x_i^{a_{i,j}}$, $1\leq j\leq m$. For every $i$
with $1\leq i\leq n$, let $a_i=\max\{a_{i,j}\mid 1\leq j\leq m\}$, and
suppose that $$T=\mathbb{K}[x_{1,1},x_{1,2},\ldots,x_{1,a_1},x_{2,1},
x_{2,2},\ldots,x_{2,a_2},\ldots,x_{n,1},x_{n,2},\ldots,x_{n,a_n}]$$ is a
polynomial ring over the field $\mathbb{K}$. Let $I^{{\rm pol}}$ be the squarefree
monomial ideal of $T$ with minimal generators $u_1^{{\rm pol}},\ldots,u_m^{{\rm pol}}$, where
$u_j^{{\rm pol}}=\prod_{i=1}^{n}\prod_{k=1}^{a_{i,j}}x_{i,k}$, $1\leq j\leq m$. The
monomial $u_j^{{\rm pol}}$ is called the {\it polarization} of $u_j$, and the ideal $I^{{\rm pol}}$
is called the {\it polarization} of $I$.
\end{dfn}

As we mentioned at the beginning of this section, for every graph $G$ and every integer $k\geq 1$, the cover ideal of $G_k$ is related to the $k$-th symbolic power of the cover ideal of $G$. This is the content of the following lemma.

\begin{lem} \label{sympol}
Let $G$ be a graph. For every integer $k\geq 1$, the ideal $(J(G)^{(k)})^{{\rm pol}}$ is the cover ideal of $G_k$.
\end{lem}

\begin{proof}
We know that polarization commutes with the intersection (see \cite[Proposition 2.3]{f1}). Therefore,$$(J(G)^{(k)})^{{\rm pol}}=\bigcap_{\{x_i, x_j\}\in E(G)}((x_i, x_j)^k)^{{\rm pol}}.$$Moreover, by \cite[Proposition 2.5]{f1}, it holds that$$((x_i, x_j)^k)^{{\rm pol}}=\bigcap_{p+q\leq k+1}(x_{i,p}, x_{j, q}).$$Thus,
$$(J(G)^{(k)})^{{\rm pol}}=\bigcap_{\{x_i, x_j\}\in E(G)} \ \bigcap_{p+q\leq k+1}(x_{i,p}, x_{j, q}).$$Therefore, $(J(G)^{(k)})^{{\rm pol}}$ is the cover ideal of $G_k$.
\end{proof}

We know from \cite[Corollary 1.6.3]{hh} that polarization preserves the graded Betti numbers. Thus a monomial ideal has linear resolution if and only if its polarization has linear resolution. In the following lemma, we show that a similar statement is true if one replaces the linear resolution by linear quotients.

\begin{lem} \label{linq}
A monomial ideal $I$ has linear quotients if and only if $I^{{\rm pol}}$ has linear quotients.
\end{lem}

\begin{proof}
We use the notations of Definition \ref{pol}.

($\Rightarrow$) By \cite[Lemma 8.2.3]{hh}, the elements of $G(I)$ can be ordered $u_1, \ldots , u_m$ such that for every pair of integers $j < i$ there exist an integer $k < i$ and a variable $x_p$ such that$$\frac{u_k}{{\rm gcd}(u_k, u_i)}=x_p \ \ \ \ \ {\rm and} \ \ \ \ \ x_p \ \ {\rm divides} \ \  \frac{u_j}{{\rm gcd}(u_j, u_i)}.$$Let $t$ be the biggest integer with $x_p^t|u_i$. It follows from the equality$$\frac{u_k}{{\rm gcd}(u_k, u_i)}=x_p$$that $x_p^{t+1}|u_k$ and $x_p^{t+2}\nmid u_k$, Therefore$$\frac{u_k^{{\rm pol}}}{{\rm gcd}(u_k^{{\rm pol}}, u_i^{{\rm pol}})}=x_{p,t+1}.$$On the other hand, $x_p$ divides$$\frac{u_j}{{\rm gcd}(u_j, u_i)}.$$by the choice of $t$, we conclude that $x_p^{t+1}|u_j$. This shows that $x_{p, t+1}$ divides$$\frac{u_j^{{\rm pol}}}{{\rm gcd}(u_j^{{\rm pol}}, u_i^{{\rm pol}})}.$$ It again follows from \cite[Lemma 8.2.3]{hh} that $I^{{\rm pol}}$ has linear quotients.

($\Leftarrow$) By \cite[Lemma 8.2.3]{hh}, the elements of $G(I^{{\rm pol}})$ can be ordered $u_1^{{\rm pol}}, \ldots , u_m^{{\rm pol}}$ such that for every pair of integers $j < i$ there exist an integer $k < i$ and a variable $x_{p,q}\in T$ such that$$\frac{u_k^{{\rm pol}}}{{\rm gcd}(u_k^{{\rm pol}}, u_i^{{\rm pol}})}=x_{p,q} \ \ \ \ \ {\rm and} \ \ \ \ \ x_{p,q} \ \ {\rm divides} \ \  \frac{u_j^{{\rm pol}}}{{\rm gcd}(u_j^{{\rm pol}}, u_i^{{\rm pol}})}.$$This shows that$$\frac{u_k}{{\rm gcd}(u_k, u_i)}=x_p \ \ \ \ \ {\rm and} \ \ \ \ \ x_p \ \ {\rm divides} \ \  \frac{u_j}{{\rm gcd}(u_j, u_i)},$$and hence, $I$ has linear quotients.
\end{proof}

We are know ready to prove the first main result of this section.

\begin{thm} \label{main}
Let $G$ be a very well-covered graph and suppose that its cover ideal $J(G)$ has linear resolution. Then
\begin{itemize}
\item[(i)] $J(G)^{(k)}$ has linear resolution, for every integer $k\geq 1$.
\item[(ii)] $J(G)^{(k)}$ has linear quotients, for every integer $k\geq 1$.
\end{itemize}
\end{thm}

\begin{proof}
Since the isolated vertices have no effect on the cover ideal, we assume that $G$ has no isolated vertex. Since $J(G)$ has linear resolution, it follows from \cite[Theorem 8.1.9]{hh} that $G$ is a Cohen--Macaulay graph. Thus, by Proposition \ref{very-well}, the graph $G_k$ is Cohen--Macaulay.

(i) Notice that \cite[Theorem 8.1.9]{hh} and Lemma \ref{sympol} imply that $(J(G)^{(k)})^{{\rm pol}}=J(G_k)$ has linear resolution. Hence, it follows from \cite[Corollary 1.6.3]{hh} that $J(G)^{(k)}$ has linear resolution.

(ii) By \cite[Theorem 1.1]{mmcrty}, we know that every Cohen--Macaulay very well-covered graph is shellabe. Therefore, $G_k$ is a shellable graph and hence,   \cite[Theorem 8.2.5]{hh} and Lemma \ref{sympol} imply that $(J(G)^{(k)})^{{\rm pol}}=J(G_k)$ has linear quotients. We know conclude from Lemma \ref{linq} that $J(G)^{(k)}$ has linear quotients.
\end{proof}

In the following corollary, we prove that the converse of Theorem \ref{main} is true for bipartite graphs.

\begin{cor} \label{linbi}
Let $G$ be a bipartite graph and $k\geq 1$ be an integer. Then $J(G)$ has linear resolution if and only if $J(G)^k$ has linear resolution.
\end{cor}

\begin{proof}
Without loss of generality, we assume that $G$ has no isolated vertex. Note that by \cite[Corollary 2.6]{grv}, the symbolic and the ordinary powers of cover ideal of bipartite graphs coincide. If $J(G)$ has linear resolution, then it follows from \cite[Theorem 8.1.9]{hh} that $G$ is Cohen--Macaulay. In particular, $G$ is unmixed and hence it is very well-covered. Therefore, the "only if" part follows from Theorem \ref{main}. To prove the "if" part assume that $V(G)=X\cup Y$ is a bipartition for the vertex set of $G$. Suppose that $X=\{x_1, \ldots, x_s\}$ and $Y=\{y_1, \ldots, y_t\}$. Clearly, we can assume that $k\geq 2$. It follows from \cite[Corollary 1.6.3]{hh} and Lemma \ref{sympol} that $(J(G)^k)^{{\rm pol}}=J(G_k)$ has linear resolution. Then \cite[Theorem 8.1.9]{hh} implies that $G_k$ is a Cohen--Macaulay graph. Notice that the set$$F=\{x_{i,j}\mid 1\leq i\leq s \ {\rm and} \ 2\leq j\leq k\}$$ is an independent subset of vertices of $G_k$. Since $G_k$ has no isolated vertex, one can easily check that$$N_{G_k}[F]=F\cup\{y_{i,j}\mid 1\leq i\leq t \ {\rm and} \ 1\leq j\leq k-1\}.$$Thus $G_k\setminus N_{G_k}[F]$ is isomorphic to $G$. This mean that ${\rm lk}_{\Delta(G_k)}F=\Delta(G)$. Since $G_k$ is Cohen--Macaulay, it follows that $G$ is Cohen--Macaulay too. Hence, \cite[Theorem 8.1.9]{hh} implies that $J(G)$ has linear resolution.
\end{proof}

Let $M$ be a finitely generated $\mathbb{Z}^n$-graded $S$-module. Let
$u\in M$ be a homogeneous element and $Z\subseteq
\{x_1,\dots,x_n\}$. The $\mathbb {K}$-subspace $u\mathbb{K}[Z]$
generated by all elements $uv$ with $v\in \mathbb{K}[Z]$ is
called a {\it Stanley space} of dimension $|Z|$, if it is a free
$\mathbb{K}[Z]$-module. Here, as usual, $|Z|$ denotes the number
of elements of $Z$. A decomposition $\mathcal{D}$ of $M$ as a
finite direct sum of Stanley spaces is called a {\it Stanley
decomposition} of $M$. The minimum dimension of a Stanley space
in $\mathcal{D}$ is called the {\it Stanley depth} of
$\mathcal{D}$ and is denoted by ${\rm sdepth} (\mathcal {D})$.
The quantity $${\rm sdepth}(M):=\max\big\{{\rm sdepth}
(\mathcal{D})\mid \mathcal{D}\ {\rm is\ a\ Stanley\
decomposition\ of}\ M\big\}$$ is called the {\it Stanley depth}
of $M$. We say that a $\mathbb{Z}^n$-graded $S$-module $M$ satisfies {\it Stanley's inequality} if $${\rm depth}(M) \leq
{\rm sdepth}(M).$$ In fact, Stanley \cite{s} conjectured that every $\mathbb{Z}^n$-graded $S$-module satisfies Stanley's inequality.
This conjecture has been recently disproved in \cite{abcj}.
However, it is still interesting to find the classes of
$\mathbb{Z}^n$-graded $S$-modules which satisfy Stanley's inequality.
For a reader friendly introduction to Stanley depth, we refer to
\cite{psty} and for a nice survey on this topic, we refer to
\cite{h}. In \cite[Corollary 3.4]{s4}, the author proves that for a bipartite graph $G$, the modules $J(G)^k$ and $S/J(G)^k$ satisfy Stanley's inequality for every integer $k\gg 0$. In the following corollary, we prove Stanley's inequality for every symbolic power of the cover ideal of Cohen--Macaulay very well-covered graphs.

\begin{cor} \label{sdepth}
Let $G$ be a very well-covered graph and suppose that its cover ideal $J(G)$ has linear resolution. Then $J(G)^{(k)}$ and $S/J(G)^{k}$ satisfy Stanley's inequality, for every integer $k\geq 1$.
\end{cor}

\begin{proof}
By Theorem \ref{main}, we know that $J(G)^{(k)}$ has linear quotients. On the other hand, it is known \cite{so} that Stanley's inequality holds true for every monomial ideal with linear quotients. Thus, $J(G)^{(k)}$ satisfies Stanley's inequality.

To prove that $S/J(G)^{k}$ satisfies Stanley's inequality, by \cite[Corollary 4.5]{ikm}, it is sufficient to prove that $T/(J(G)^{(k)})^{{\rm pol}}$ satisfies Stanley's inequality (where $T$ is the new polynomial ring). By assumption and \cite[Theorem 8.1.9]{hh}, we conclude that $G$ is Cohen--Macaulay graph. Again, we can assume that $G$ has no isolated vertex. Thus, Proposition \ref{very-well} implies that $G_k$ is a Cohen--Macaulay very well-covered graph. It then follows from \cite[Theorem 1.1]{mmcrty} that $G_k$ is a vertex decomposable graph. Now, It follows from \cite{sh} that $T/J(G_k)$ satisfies Stanley's inequality. Finally, the assertion follows from Lemma \ref{sympol}.
\end{proof}

In \cite[Theorem 3.2]{cpsty}, the authors prove that the depth of the symbolic powers of the cover ideal of a bipartite graph is a non-increasing sequence. In the following theorem \ref{depthsym}, we prove the same for every very well-covered graphs. This, in particular implies that the sequence $\{{\rm depth}(S/J(G)^{(k)})\}_{k=1}^{\infty}$ is convergent.

\begin{thm} \label{depthsym}
Let $G$ be a very well-covered graph. Then for every integer $k\geq 1$, we have$${\rm depth}(S/J(G)^{(k)})\geq {\rm depth}(S/J(G)^{(k+1)}).$$
\end{thm}

\begin{proof}
We may assume that $G$ has no isolated vertex. Using Auslander--Buchsbaum Formula, it is enough to prove that$${\rm pd}(S/J(G)^{(k)})\leq {\rm pd}(S/J(G)^{(k+1)}).$$Since projective dimension is preserved by polarization \cite[Corollary 1.6.3]{hh}, we need to prove that$${\rm pd}(T/(J(G)^{(k)})^{{\rm pol}})\leq {\rm pd}(T/(J(G)^{(k+1)})^{{\rm pol}}),$$(where $T$ is the new polynomial ring which contains both $(J(G)^{(k)})^{{\rm pol}}$ and $(J(G)^{(k+1)})^{{\rm pol}}$). It follows from Lemma \ref{sympol} that $(J(G)^{(k)})^{{\rm pol}}=J(G_k)$. Then Terai's theorem \cite[Theorem 8.1.10]{hh} implies that ${\rm pd}(T/J(G_k))={\rm reg}(T/I(G_k))+1$. Since $G$ is a very well-covered graph, we conclude from Proposition \ref{very-well} implies that $G_k$ is very well-covered too. Then \cite[Theorem 1.3]{mmcrty} implies that ${\rm reg}(T/I(G_k))={\rm indmatch}(G_k)$. Thus, $${\rm pd}(T/(J(G)^{(k)})^{{\rm pol}})={\rm indmatch}(G_k)+1.$$Similarly, $${\rm pd}(T/(J(G)^{(k+1)})^{{\rm pol}})={\rm indmatch}(G_{k+1})+1.$$Since $G_k$ is an induced subgraph of $G_{k+1}$, it follows that$${\rm indmatch}(G_k)\leq {\rm indmatch}(G_{k+1})$$ and this completes the proof.
\end{proof}


\section{Bipartite graphs} \label{sec3}

In this section, we determine a linear upper bound for the regularity of powers of cover ideal of bipartite graphs. For a monomial ideal $I$, let ${\rm deg}(I)$ denote the maximum degree of elements of $G(I)$. Thus in particular, ${\rm deg}(J(G))$ is the cardinality of the largest minimal vertex cover of the graph $G$. It is clear that for every integer $k\geq 1$ and every graph $G$, the regularity of $S/J(G)^k$ is at least $k{\rm deg}(J(G))-1$. In Theorem \ref{regbi}, we prove that the regularity of $S/J(G)^k$ can not be much larger than $k{\rm deg}(J(G))-1$, when $G$ is a bipartite graph. We first need the following simple lemma.

\begin{lem} \label{int}
Assume that $I\subseteq S$ is a monomial ideal. Let $S'=\mathbb{K}[x_2, \ldots, x_n]$ be the polynomial ring obtained from $S$ by deleting the variable $x_1$ and set $I'=I\cap S'$. Then ${\rm reg}(I')\leq {\rm reg}(I)$.
\end{lem}

\begin{proof}
Using polarization, we can assume that $I$ is a squarefree monomial ideal. Then the assertion follows immediately from Hochster's formula \cite[Theorem 8.1.1]{hh}.
\end{proof}

The following lemma is a consequence of Lemma \ref{int}.

\begin{lem} \label{colon}
Let $I$ be a monomial ideal of $S$. Then for every monomial $u\in S$, we have ${\rm reg}(S/(I:u))\leq {\rm reg}(S/I)$.
\end{lem}

\begin{proof}
Clearly, we can assume that $u$ is a variable, say $u=x_1$. By applying \cite[Corollary 18.7]{p'} on the exact sequence
\[
\begin{array}{rl}
0\longrightarrow S/(I:x_1)(-1)\longrightarrow S/I\longrightarrow S/(I, x_1)
\longrightarrow 0,
\end{array}
\]
we obtain that
\[
\begin{array}{rl}
{\rm reg}(S/(I:x_1))+1\leq \max\{{\rm reg}(S/I), {\rm reg}(S/(I,x_1))+1\}\leq {\rm reg}(S/I)+1,
\end{array}
\]
Where the last inequality follows from Lemma \ref{int}.
\end{proof}

We are now ready to prove the main result of this section.

\begin{thm} \label{regbi}
Let $G$ be a bipartite graph with $n$ vertices. Then for every integer $k\geq 1$, we have$${\rm reg}(S/J(G)^k)\leq k{\rm deg}(J(G))+{\rm reg}(S/J(G))-1.$$
\end{thm}

\begin{proof}
Assume that $V(G)=U\cup W$ is a bipartition for the vertex set of $G$. Without loss of generality, we may assume that $U=\{x_1, \ldots, x_t\}$ and $W=\{x_{t+1}, \ldots, x_n\}$, for some integer $t$ with $1\leq t\leq n$. Let $m$ be the number of edges of $G$. We prove the assertions by induction on $m+k$. Clearly, we can assume that $G$ has no isolated vertex.

There is nothing to prove for $k=1$. If $m=1$, then $J(G)=(x_1,y_1)$. Hence ${\rm deg}(J(G))=1$ and ${\rm reg}(S/J(G)^k)=k-1$. Thus, the desired inequality is true for $m=1$. Therefore, assume that $k,m\geq 2$. Let $S_1=\mathbb{K}[x_2, \ldots, x_n]$ be the polynomial ring obtained from $S$ by deleting the variable $x_1$ and consider the ideals $J_1=J(G)^k\cap S_1$ and
$J_1'=(J(G)^k:x_1)$. It follows from \cite[Lemma 2.10]{dhs} that
\[
\begin{array}{rl}
{\rm reg}(S/J(G)^k)\leq \max \{{\rm reg}_{S_1}(S_1/J_1), {\rm reg}_S(S/J_1')+1\},
\end{array} \tag{\dag} \label{dag}
\]

Notice that $J_1=(J(G)\cap S_1)^k$. Hence, by Lemma \cite[Lemma 2.2]{s4}, there exists a monomial $u_1\in S_1$ with ${\rm deg}(u_1)={\rm deg}_G(x_1)$ such that $J(G)\cap S_1=u_1J(G\setminus N_G[x_1])S_1$ and thus, $J_1=u_1^kJ(G\setminus N_G[x_1])^kS_1$. Notice that if $C$ is a minimal vertex cover of $G\setminus N_G[x_1]$, then $C\cup N_G(x_1)$ is a minimal vertex cover of $G$. This shows that ${\rm deg}(J(G\setminus N_G[x_1]))+{\rm deg}_G(x_1)\leq {\rm deg}(J(G))$. On the other hand, Lemma \ref{int} implies that ${\rm reg}(J(G)\cap S_1)\leq {\rm reg}(J(G))$ and therefore,$${\rm reg}_{S_1}(S_1/J(G\setminus N_G[x_1])S_1)\leq {\rm reg}(S/J(G))-{\rm deg}(u_1).$$ Hence, by the induction hypothesis we conclude that
\begin{align*}
& {\rm reg}_{S_1}(S_1/J_1)={\rm reg}_{S_1}(S_1/J(G\setminus N_G[x_1])^kS_1)+k{\rm deg}(u_1)\\
& \leq k{\rm deg}(J(G\setminus N_G[x_1]))+{\rm reg}(S/J(G\setminus N_G[x_1]))-1+k{\rm deg}_G(x_1)\\
& \leq k({\rm deg}(J(G))-{\rm deg}_G(x_1))+{\rm reg}(S/J(G))-{\rm deg}_G(x_1)-1+k{\rm deg}_G(x_1)\\
& \leq k{\rm deg}(J(G))+{\rm reg}(S/J(G))-1.
\end{align*}
Thus, using the inequality (\ref{dag}), it is enough to prove that ${\rm reg}_S(S/J_1')\leq k{\rm deg}(J(G))+{\rm reg}(S/J(G))-2$.

For every integer $i$ with $2\leq i\leq t$, let $S_i=\mathbb{K}[x_1, \ldots, x_{i-1}, x_{i+1}, \ldots, x_n]$ be the polynomial ring obtained from $S$ by deleting the variable $x_i$ and consider the ideals $J_i'=(J_{i-1}':x_i)$ and $J_i=J_{i-1}'\cap S_i$.

{\bf Claim.} For every integer $i$ with $1\leq i\leq t-1$ we have$${\rm reg}(S/J_i')\leq \max\{k{\rm deg}(J(G))+{\rm reg}(S/J(G))-2, {\rm reg}_S(S/J_{i+1}')+1\}.$$

{\it Proof of the Claim.} For every integer $i$ with $1\leq i\leq t-1$, we know from \cite[Lemma 2.10]{dhs} that
\[
\begin{array}{rl}
{\rm reg}(S/J_i')\leq \max \{{\rm reg}_{S_{i+1}}(S_{i+1}/J_{i+1}), {\rm reg}_S(S/J_{i+1}')+1\}.
\end{array} \tag{$\ast$} \label{aast}
\]

Notice that for every integer $i$ with $1\leq i\leq t-1$, we have $J_i'=(J(G)^k:x_1x_2\ldots x_i)$. Thus $J_{i+1}=J_i'\cap S_{i+1}=((J(G)^k\cap S_{i+1}):_{S_{i+1}}x_1x_2\ldots x_i)$. Hence, it follows from Lemma \ref{colon} that

\[
\begin{array}{rl}
{\rm reg}_{S_{i+1}}(S_{i+1}/J_{i+1})\leq {\rm reg}_{S_{i+1}}(S_{i+1}/(J(G)^k\cap S_{i+1})).
\end{array} \tag{$\ast\ast$} \label{aaaast}
\]

By Lemma \cite[Lemma 2.2]{s4}, we conclude that there exists a monomial $u_{i+1}\in S_{i+1}$, with ${\rm deg}(u_{i+1})={\rm deg}_G(x_{i+1})$ such that $J(G)\cap S_{i+1}=u_{i+1}J(G\setminus N_G[x_{i+1}])S_{i+1}$. Therefore$$J(G)^k\cap S_{i+1}=u_{i+1}^kJ(G\setminus N_G[x_{i+1}])^kS_{i+1}.$$Notice that if $C$ is a minimal vertex cover of $G\setminus N_G[x_{i+1}]$, then $C\cup N_G(x_{i+1})$ is a minimal vertex cover of $G$. This shows that ${\rm deg}(J(G\setminus N_G[x_{i+1}]))+{\rm deg}_G(x_{i+1})\leq {\rm deg}(J(G))$. On the other hand, Lemma \ref{int} implies that ${\rm reg}(J(G)\cap S_{i+1})\leq {\rm reg}(J(G))$ and therefore,$${\rm reg}_{S_{i+1}}(S_{i+1}/J(G\setminus N_G[x_{i+1}])S_{i+1})\leq {\rm reg}(S/J(G))-{\rm deg}(u_{i+1}).$$ Hence, by the induction hypothesis we conclude that
\begin{align*}
& {\rm reg}_{S_{i+1}}(S_{i+1}/(J(G)^k\cap S_{i+1}))={\rm reg}_{S_{i+1}}(S_{i+1}/J(G\setminus N_G[x_{i+1}])^kS_{i+1})+k{\rm deg}(u_{i+1})\\
& \leq k{\rm deg}(J(G\setminus N_G[x_{i+1}]))+{\rm reg}(S/J(G\setminus N_G[x_{i+1}]))-1+k{\rm deg}_G(x_{i+1})\\
& \leq k({\rm deg}(J(G))-{\rm deg}_G(x_{i+1}))+{\rm reg}(S/J(G))-{\rm deg}_G(x_{i+1})-1+k{\rm deg}_G(x_{i+1})\\
& \leq k{\rm deg}(J(G))+{\rm reg}(S/J(G))-2.
\end{align*}
Now the claim follows by inequalities (\ref{aast}) and (\ref{aaaast}).

Now, $J_t'=(J(G)^k:x_1x_2\ldots x_t)$ and hence, \cite[Lemma 3.2]{s4} implies that $J_t'=J(G)^{k-1}$ and thus, by induction hypothesis we conclude that ${\rm reg}(S/J_t')\leq (k-1){\rm deg}(J(G))+{\rm reg}(S/J(G))-1$. Therefore, using the claim repeatedly, implies that
\begin{align*}
& {\rm reg}(S/J_1')\leq \max\{k{\rm deg}(J(G))+{\rm reg}(S/J(G))-2, {\rm reg}_S(S/J_t')+t-1\}\\
& \leq \max\{k{\rm deg}(J(G))+{\rm reg}(S/J(G))-2, (k-1){\rm deg}(J(G))+{\rm reg}(S/J(G))+t-2\}.
\end{align*}
Note that $t=\mid U\mid\leq {\rm deg}(J(G))$. Thus, the above inequalities imply that ${\rm reg}(S/J_1')\leq k{\rm deg}(J(G))+{\rm reg}(S/J(G))-2$. This completes the proof of the theorem.
\end{proof}





\end{document}